\documentclass[12pt]{article}
\usepackage{amscd,amssymb,amsmath,verbatim,amsthm, graphicx}

\usepackage{color}

\newcommand{\bS}{\mathbb S}

\newcommand{\ZZ}{\mathbb Z}

\newcommand{\frakp}{\mathfrak p}

\newtheorem{theorem}{Theorem}

\newtheorem{lemma}{Lemma}
\theoremstyle{definition}
\newtheorem{definition}{Definition}
\theoremstyle{remark}
\newtheorem*{remark}{Remark}

\title{Rigidity of a family of spherical conical metrics} 
\author{Xuwen Zhu\\UC Berkeley}

\date{}

\begin{document}

\maketitle

\begin{abstract}
We study the deformation of spherical conical metrics with at least 
some of the cone angles larger than $2\pi$. We show in this note via synthetic geometry that 
for one family of such metrics, there is local rigidity in the choice of cone positions if angles are fixed. This gives an evidence of the analytic obstruction considered in recent works of Mazzeo and author~\cite{MZ, MZ2}. 
\end{abstract}

\section{Introduction}

The study of constant curvature metrics with singularities has seen a long and rich history, where a lot of interesting questions are still not completely answered. 
Among them is the following singular uniformization question:
given a compact Riemann surface $M$, a collection
of distinct points $\frakp = \{p_1, \ldots, p_k\} \subset M$ and a collection of positive real numbers
$\beta_1, \ldots, \beta_k$, is it possible to find a metric $g$ on $M$ with constant
curvature and with conic singularities with prescribed cone angles $2\pi \beta_j$ at
the points $p_j$?  Here the sign of its curvature is determined by the `conic' Gauss-Bonnet formula
\begin{equation}\label{e:GB}
\frac{1}{2\pi}\int_M K\, dA = \chi(M, \vec \beta): = \chi(M) + \sum_{j=1}^k (\beta_j-1).
\end{equation}

When $\chi(M, \vec \beta)\leq 0$, the existence and uniqueness of such solutions are proved by McOwen~\cite{McOwen}. In the spherical $K=1$ case with all cone angles less than $2\pi$, Troyanov \cite{Tr} gave a set of linear inequalities on the $\beta_j$'s
which are necessary and sufficient for existence; Luo and Tian \cite{LT} later proved uniqueness of the solution
in this angle regime. 
For all the above cases, there is no restriction on the position of cone points. Deformation theory for these cases has been studied by Mazzeo and Weiss~\cite{MW} and it is shown that the metrics have smooth dependence on cone angles and positions. 

When $K=1$ with at least some of the cone angles bigger than $2\pi$, the story is much more complicated. Recently Mondello and Panov~\cite{MP} discovered that when $M=\bS^{2}$ the cone angles are constrained by a set of linear inequalities 
\begin{equation}\label{MonPan}
 d_1( \vec \beta - \vec 1, \ZZ^k_{\mathrm{odd}}) \geq 1,
\end{equation}
and showed the existence when the strict inequality holds; the boundary cases have been considered in~\cite{Dey, Ka, Ere1}. 
The same two authors~\cite{MP2} also showed that 
when $M\neq \bS^{2}$, the condition $\chi(M, \vec\beta)>0$ is sufficient for existence.
 In either cases, one is unable to specify the marked conformal class, i.e., the location of the points 
$\frakp$.

In this paper we consider the deformation of the following metrics with four conical points on $\bS^{2}$: 
\begin{equation}\label{e:beta}
2\pi\vec \beta=(\alpha, \beta, \alpha+\beta, 4\pi),\ \alpha, \beta, \alpha+\beta\notin 2\pi\ZZ,
\end{equation}
and show that there is local rigidity in the location of cone points. Note here~\eqref{e:beta} satisfies the equality in~\eqref{MonPan}, and such angle combinations lie on the codimension-two boundary of the admissible region.

For any fixed $\vec \beta$ satisfying~\eqref{e:beta}, there exists a real one-parameter family of cone point positions $\{\frakp_{t}, 0<t<\pi\}$ on $\bS^{2}$ such that there exists a spherical conical metric $g_{t}$ with angles $2\pi\vec\beta$ on $(\bS^{2}, \frakp_{t})$. The geometric realization of such metrics is obtained by gluing two spherical footballs of angles $\alpha$ and $\beta$ along part of a meridian with one end at the south pole (see Figure~\ref{pic0}) where $t$ parametrizes the length of the cut. 

\begin{figure}[!]
\centering
\includegraphics[width=0.7\textwidth]{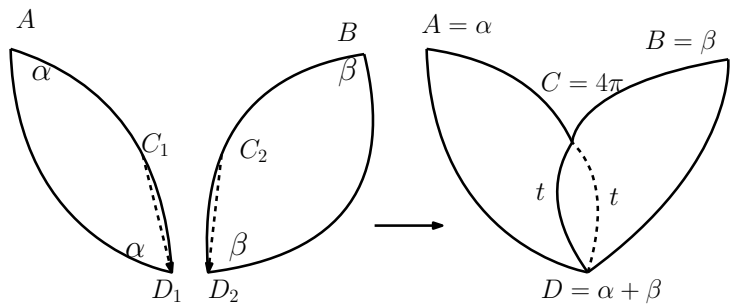}
 \caption{A spherical metric with cone angles $\alpha, \beta, \alpha+\beta, 4\pi$ can be obtained by gluing together two spherical footballs along a slit of length $t$. We cut open each football along a geodesic of length $t$, and glue points $C_{1}$ with $C_{2}$ to get a cone point $C$ with angle $4\pi$, $D_{1}$ with $D_{2}$ to get the point $D$ with angle $\alpha+\beta$, and glue the two pairs of geodesics between $C_{1}D_{1}$ and $C_{2}D_{2}$.}
\label{pic0}
\end{figure}

In this note we approach the rigidity of such metrics via synthetic geometry. We decompose the surface along geodesics to get four spherical triangles, see Figure~\ref{pic7}. 
\begin{figure}[!h]
\includegraphics[width=\textwidth]{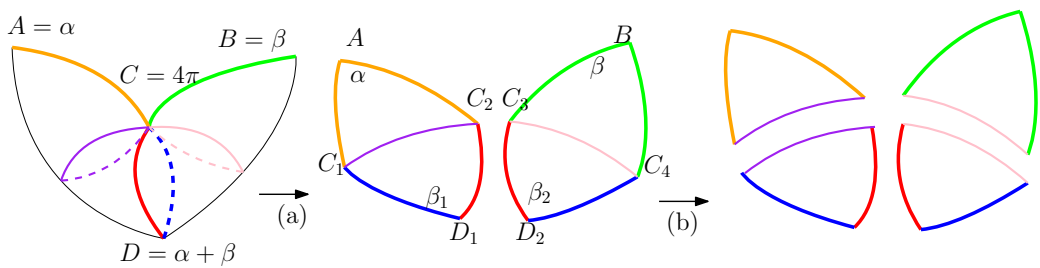}
 \caption{The process of cutting the surface into four spherical geodesic triangles (here geodesics with the same color glue together): (a) cut along geodesics $AC, BC$ and two geodesics connecting $CD$ to get two surfaces, each with four sides, where $C_{1}-C_{4}$ glue back to cone point $C$ and $D_{1}-D_{2}$ glue to cone point $D$; (b) further cut along two geodesics $C_{1}C_{2}$ and $C_{3}C_{4}$ to obtain four triangles.}
 \label{pic7}
\end{figure}
\begin{definition}
A \textbf{triangulated metric} is defined to be a spherical conical metric with the same geodesic decomposition as in Figure~\ref{pic7} (not necessarily with the same cone angles or geodesic lengths). For two triangulated metrics $h$ and $g$, $h$ is called \textbf{closed to $g$} if lengths of all boundaries of the four geodesic triangles of $h$ are close to the corresponding ones of $g$.
\end{definition}

In most cases, a small perturbation of metrics preserve the geodesic decomposition, and hence we expect all nearby spherical conical metrics  to be triangulated and close to each other in the sense defined above. 

The space of triangulated metrics are parametrized by six (independent) lengths as long as they satisfy spherical triangle inequalities, see Figure~\ref{pic7} for the color-coded geodesic pairs.
If we restrict to such metrics with four fixed cone angles, which imposes four equations on six parameters, then the space of all such metrics is two dimension for a generic angle set $\vec \beta$.  We remark here that the conformal class of four marked points on sphere is determined by the (complex) cross ratio which also gives a (real) two-dimension space. Therefore in general we expect that the neighborhood of triangulated metrics contains all possible perturbations in the usual sense.

We will show that for $\vec \beta$ as in~\eqref{e:beta}, the glued footballs are the only possible triangulated metrics under perturbation, hence all such metrics form a one-dimensional space.
And this gives the local rigidity in the geometric sense.

\begin{theorem}
For any fixed $\vec \beta$ as in~\eqref{e:beta} and $t\in (0, \pi)$, if $(\bS^{2}, h)$ is a  triangulated spherical metric with the same cone angles and $h$ is close to $g_{t}$, then $h$ is isometric to $g_{s}$ for some $s$. 
\end{theorem}

Such geometric rigidity has appeared in the case of a spherical football with noninteger cone angles. 
By the classical proof of Troyanov~\cite{Tr2}, the only possible configuration in this case is when all the geodesics connecting two cone points are of length $\pi$. By the perturbation argument in~\cite{MW}, the spherical football is the only metric with rigidity when all cone angles are less than $2\pi$.
However, when there are more than three cone points with some of angles bigger than $2\pi$, such rigidity is far from clear. This note intends to give a family of explicit metrics in this regime.

The angle combination~\eqref{e:beta} was discussed by Chen, Wang, Wu and Xu~\cite[Example 4.7]{CWWX} and it was shown that if such a metric has reducible monodromy, then the positions of cone points with the first three angles $\alpha, \beta, \alpha+\beta$ determine the position of the 4th point. And there is a real 1-parameter family of such metrics where the parameter comes from the scaling of the character 1-form. Since such angle combination lies on the boundary of the admissible region of~\eqref{MonPan}, 
from~\cite[Corollary 2.25(I)]{MP}, this implies that the monodromy of such conical metrics is necessarily reducible. In particular, this implies that there is rigidity in the conformal class of the quadruply punctured sphere for this angle combination. 

This note intends to give a geometric realization of such rigidity, and to our knowledge, this is the first proof using elementary spherical geometry.  We also hope to give more insight in understanding the rigidity of cone metrics as solutions to the curvature equation with prescribed singularities.

We observe that if we write the metrics on a football with angle $2\pi\alpha$ in geodesic coordinates as 
$$
dr^{2}+\alpha^{2}\sin^{2}r d\theta^{2},
$$
then $\cos r$ is an eigenfunction of its Laplacian with eigenvalue 2. Moreover, when two footballs (not necessarily with same angles) glue together, this eigenfunction also glue to give a global one.
Therefore the metrics we consider here all satisfy the condition that number $2$ lies in the Friedrichs extension of the Laplacian, hence from~\cite{MZ2} such partial rigidity in cone positions is expected as a result of the obstruction in solving the curvature equation.

The study of constant curvature conical metrics has seen a lot of recent development.  One approach is through complex analysis, see~\cite{EG, EGT, EGT2, UY}. For metrics with special monodromy which is of particular interest of this note, see the works of Xu and collaborators~\cite{CWWX, SCLX, SLX} and Eremenko~\cite{Ere1}.  We also 
mention here the variational approach by Malchiodi and collaborators~\cite{BMM, BM, CM2, Carlotto} and the Leray-Schauder degree counting method by Chen and Lin~\cite{CL2, CL3}. We refer the readers to~\cite{MZ2} for a more comprehensive overview. We also mention here another type of closely related objects called HCMU metrics. This exhibit a similar obstruction in existence~\cite{Chen}, and geodesic decomposition was used to analyze such metrics~\cite{CCW}. 

This paper is organized as follows. In~\S\ref{s:equalangle} we consider the case when $\alpha=\beta$, for which the computation is simpler but keeps the essential feature of the proof. In~\S\ref{s:full} we give the proof of the general case.

\noindent\textbf{Acknowledgement: } The author would like to thank Rafe Mazzeo and Bin Xu for many useful discussions and suggestions.

\section{The case $\alpha=\beta$}\label{s:equalangle}
We start with the example when $\alpha=\beta$ to simplify the computation. We first show that by assuming some symmetry, the only possibility would be the glued footballs.
\begin{lemma}\label{lemma1}
If $\beta_{1}=\beta_{2}(=\beta)$, then $h$ is equal to $g_{s}$ for some $s$.  
\end{lemma}
\begin{proof}
In the proof we focus on the intermediate step of the triangulation given in Figure~\ref{pic7}, where we have two pieces, each of which is a spherical domain with four sides. 
First connect $C_{1}C_{2}$ and $C_{3}C_{4}$ by geodesics and consider their lengths, see Figure~\ref{pic2}. 
\begin{figure}[h]
\centering
\includegraphics[width=0.7\textwidth]{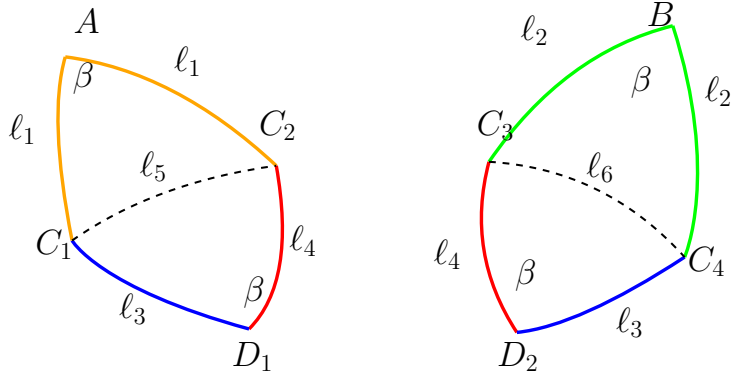}
 \caption{When $\beta_{1}=\beta_{2}$, we show that the two pieces should be two identical spherical bigons.}
 \label{pic2}
\end{figure}

If $\beta_{1}=\beta_{2}=\beta$, then $\ell_{5}=\ell_{6}$ by spherical cosine law
\begin{equation}\label{e:cosine}
\cos \ell_{5}=\cos\ell_{3}\cos \ell_{4}+\sin \ell_{3}\sin\ell_{4}\cos\beta=\cos \ell_{6}.
\end{equation} 
Hence for the two spherical triangles $A_{1}C_{1}C_{2}$ and $BC_{3}C_{4}$, there are two possibilities: (a) they are identical; (b) they are not identical but piece together to a spherical bigon. 

We first show that case (b) is not possible. If $\ell_{1}\neq \ell_{2}$, then we have $\sin \ell_{1}\neq 1$  since $\ell_{1}+\ell_{2}=\pi$. Applying spherical cosine law again, we have
$$
\cos \ell_{5}=\cos \ell_{6}=(\cos\ell_{i})^{2}+(\sin \ell_{i})^{2}\cos\beta=1+(\cos \beta-1)(\sin \ell_{i})^{2},  i=1,2.
$$
Therefore we know that 
\begin{equation}\label{e:con1}
\frac{\cos \ell_{5}-1}{\cos \beta-1}<1.
\end{equation}
In this case we also have $AC_{1}C_{2}=AC_{2}C_{1}=\pi-BC_{3}C_{4}=\pi-BC_{4}C_{3}\neq \pi/2$.
Since $AC_{2}D_{1}+AC_{1}D_{1}+ BC_{3}D_{2}+ BC_{4}D_{2}=4\pi$ and the two spherical triangles $C_{1}C_{2}D_{1}$ and $C_{3}C_{4}D_{2}$ are identical, we have 
$$
C_{1}C_{2}D_{1}+C_{2}C_{1}D_{1}=\pi(=C_{3}C_{4}D_{2}+C_{4}C_{3}D_{2}).
$$
Then denoting angle $C_{2}C_{1}D_{1}=\alpha$, by cosine law for the triangle $C_{1}C_{2}D_{1}$ we have 
$$
\cos \beta=-\cos \alpha \cos (\pi-\alpha) + \sin \alpha\sin(\pi-\alpha)\cos \ell_{5}=1+(\cos \ell_{5}-1)(\sin \alpha)^{2}
$$
hence 
$$
\frac{\cos \beta-1}{\cos \ell_{5}-1}\leq 1
$$
which contradicts~\eqref{e:con1}. Therefore this show that case (b) is not possible.

We now consider case (a) and show this gives $h=g_{s}$ for some $s$. Since $\ell_{1}=\ell_{2}$, the two 4-sided surfaces in Figure~\ref{pic2} are identical. So we have the following relations of angles: 
$$AC_{2}D_{1}=BC_{3}D_{2}, \ AC_{1}D_{1}=BC_{4}D_{2}.$$ 
Since these four angles add up to $4\pi$, this means that 
\begin{equation}\label{e:2pi}
AC_{2}D_{1}+AC_{1}D_{1}=2\pi, \ BC_{3}D_{2}+ BC_{4}D_{2}=2\pi. 
\end{equation} 
Note that when $AC_{2}D_{1}=AC_{1}D_{1}=BC_{3}D_{2}= BC_{4}D_{2}=\pi$, and $\ell_{3}=\ell_{4}=\pi-\ell_{1}$,  we have two bigons which satisfy~\eqref{e:2pi}.  
We now prove that this is the only possibility. 
If $\ell_{3}\neq \ell_{4}$, then the sum of the two angles $C_{1}C_{2}D_{1}+C_{2}C_{1}D_{1}$ will be either less than or bigger than the symmetric case by using Lemma~\ref{symmetry}, which will make the total sum $AC_{2}D_{1}+AC_{1}D_{1}<2\pi$ or $>2\pi$. Therefore we must have $\ell_{3}=\ell_{4}$. This way we get two bigons, which piece together to give $g_{s}$ where $s=\ell_{3}=\ell_{4}$. 
\end{proof}

We next show that if $\beta_{1}\neq \beta_{2}$, then the angle combination $\vec \beta$ cannot be realized. 
\begin{lemma}\label{lemma2}
If $\beta_{1}\neq \beta_{2}$, then $AC_{2}D_{1}+AC_{1}D_{1}+ BC_{3}D_{2}+ BC_{4}D_{2}\neq 4\pi$ in Figure~\ref{pic7}. 
\end{lemma}
\begin{proof}
Without loss of generality, we assume $\beta_{1}=\beta-2\epsilon, \beta_{2}=\beta+2\epsilon, \epsilon>0$. 

With these assumptions, for any fixed $\ell_{1}$ and $\ell_{2}$, we will show that no choice of $\ell_{3}$ and $\ell_{4}$ would give the combination of $4\pi$. In fact, we will show that the total sum of the four angles is strictly less than (or bigger than) $4\pi$ depending on the length of $\ell_{1}$ and $\ell_{2}$.  

Since we are considering perturbation of a metric $g_{t}$, $\ell_{1}$ and $\ell_{2}$ are both close to $\pi-t$. Therefore unless $t=\pi/2$, we can assume $(\ell_{1}-\pi/2)(\ell_{2}-\pi/2)>0$. We will discuss the case of $t=\pi/2$ near the end of this proof.

We first look at the case when $\ell_{1}, \ell_{2}<\pi/2$. In this case $\ell_{3}, \ell_{4}>\pi/2$. We will show that even in the case when $\ell=\ell_{3}=\ell_{4}>\pi/2$ (which will give the maximal possible angle sum by Lemma~\ref{symmetry}), the total sum of the four angles in question is still less than $4\pi$. 

When $\ell_{3}=\ell_{4}$, the two pieces each have a $\ZZ_{2}$ symmetry, therefore we can consider a half of each piece, which give us two spherical triangles. See picture~\ref{pic5}.  

\begin{figure}[h]
\centering
\includegraphics[width=0.5\textwidth]{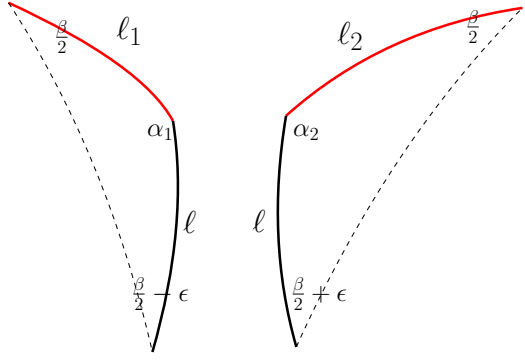}
 \caption{The two half pieces}
 \label{pic5}
\end{figure}

By spherical sine rule we have
\begin{equation}\label{e:sine}
\frac{\sin \ell_{1}}{\sin \frac{\beta_{1}}{2}}=\frac{\sin\ell}{\sin \frac{\beta}{2}}, \ \frac{\sin \ell_{2}}{\sin \frac{\beta_{2}}{2}}=\frac{\sin\ell}{\sin \frac{\beta}{2}}
\end{equation}
In particular, since $\sin \frac{\beta_{2}}{2}>\sin\frac{\beta}{2}>\sin\frac{\beta_{1}}{2}$ this implies that $\ell_{1}<\pi-\ell<\ell_{2}$. By Napier's analogies, 
\[
\cot (\frac{1}{2}\alpha_{1})=\frac{\tan (\frac{\beta}{2}-\frac{\beta_{1}}{2})\sin [\frac{1}{2}(\ell_{1}+\ell)]}{\sin [\frac{1}{2}(\ell-\ell_{1})]}, \ \cot (\frac{1}{2}\alpha_{2})=\frac{\tan (\frac{\beta}{2}-\frac{\beta_{2}}{2})\sin [\frac{1}{2}(\ell_{2}+\ell)]}{\sin [\frac{1}{2}(\ell-\ell_{2})]}.
\]
Since $\frac{\beta}{2}-\frac{\beta_{1}}{2}=-(\frac{\beta}{2}-\frac{\beta_{2}}{2})=\epsilon$, if we can show 
\[
\frac{\sin [\frac{1}{2}(\ell_{1}+\ell)]}{\sin [\frac{1}{2}(\ell-\ell_{1})]}>\frac{\sin [\frac{1}{2}(\ell_{2}+\ell)]}{\sin [\frac{1}{2}(\ell-\ell_{2})]}
\]
then this would imply $\alpha_{1}+\alpha_{2}<2\pi$, or equivalently $AC_{2}D_{1}+AC_{1}D_{1}+ BC_{3}D_{2}+ BC_{4}D_{2}<4\pi$. Now to show the above inequality holds, we just need to expand it and see that after cancellation it is equivalent to 
\begin{equation}\label{e:inequality}
\sin \ell \cos \ell \sin [\frac{1}{2}(\ell_{1}-\ell_{2})]>0
\end{equation}
which is true by our assumption $\ell>\pi/2$ and $\ell_{1}<\ell_{2}$.

Now for the other case when $\ell_{1}, \ell_{2}>\pi/2$. In this case we show that even in the minimum $\ell=\ell_{3}=\ell_{4}<\pi/2$, the total sum is still greater than $4\pi$. Note here we have $\ell_{1}>\pi-\ell>\ell_{2}$. To show $\alpha_{1}+\alpha_{2}>2\pi$, we need to have
\[
-\frac{\sin [\frac{1}{2}(\ell_{1}+\ell)]}{\sin [\frac{1}{2}(\ell-\ell_{1})]}>-\frac{\sin [\frac{1}{2}(\ell_{2}+\ell)]}{\sin [\frac{1}{2}(\ell-\ell_{2})]}
\]
and we arrive at the same inequality (note here $\sin [\frac{1}{2}(\ell-\ell_{i})]<0$)
\[
\sin \ell \cos \ell \sin [\frac{1}{2}(\ell_{1}-\ell_{2})]>0
\]
which holds this time since $\ell<\pi/2$ and $\ell_{1}-\ell_{2}>0$. And this shows even in the minimum case we still have $\alpha_{1}+\alpha_{2}>2\pi$.
Therefore the sum of the four angles is always bigger than $4\pi$ in this case, therefore cannot be realized. 

Finally we show that in the case when $t=\pi/2$, there is still no possible choice of perturbation. It is easy to show  that $\ell_{1}$ cannot be equal to $\pi/2$. If $\ell_{1}=\pi/2$ which implies $\cos \ell_{5}=\cos \beta$, then by applying cosine law to the triangle $C_{1}C_{2}D_{1}$ one gets $\cos \beta=(\cos \ell_{1})^{2}+(\sin \ell_{1})^{2}\cos (\beta-2\epsilon)$ and this will force $(\cos \ell_{1})^{2}<0$, so this is impossible. Therefore either $\ell_{1}<\pi/2$ or $\ell_{1}>\pi/2$. Then using the same sine rule~\eqref{e:sine}, we will have either $\ell_{1}<\pi-\ell <\ell_{2}\leq \pi/2$ or $\ell_{1}>\pi-\ell>\ell_{2}\geq \pi/2$. In particular we also have $\ell\neq \pi/2$. Then the same computation for~\eqref{e:inequality} would follow on both cases and we see that either $\alpha_{1}+\alpha_{2}<2\pi$ or $\alpha_{1}+\alpha_{2}>2\pi$.


This completes the proof.

\end{proof}

We now prove the fact we have used in the proofs above, that for a spherical triangle, if the length of one side and its opposite angle are fixed, then the extremal of the sum of its three angles is achieved by an isosceles spherical triangle.  Note that for a fixed $\ell$ and $\beta$ with $\cos \ell\neq \cos \beta$, there are two different isosceles triangles, one with angle $\alpha<\pi/2$ and the other with $\alpha'>\pi/2$. See picture~\ref{pic6}. 

\begin{figure}[h]
\centering
\includegraphics[width=0.4\textwidth]{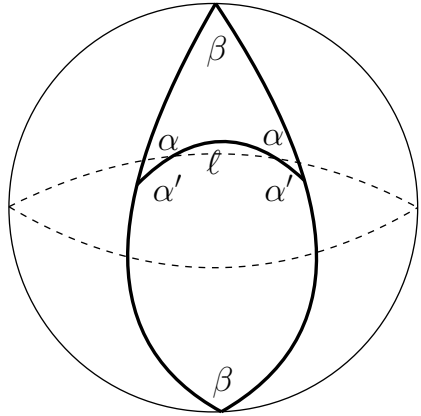}
 \caption{Two spherical isosceles triangles sharing one side and having the same opposite angle}
 \label{pic6}
\end{figure}

\begin{lemma}\label{symmetry}
For the spherical triangle $ABC$ in Figure~\ref{pic4}, if angle $ACB=\beta$ and length $AB=\ell\notin 2\pi \ZZ$ are fixed and they satisfy $\cos \ell\neq \cos \beta$, then the sum of two angles $CAB+CBA$ reaches its minimum and maximum
when $AC=BC$. Denoting angle $CAB=\alpha$, the minimum is achieved when $\alpha<\pi/2$, and the maximum is achieved when $\alpha>\pi/2$.
\end{lemma}

\begin{figure}[h]
\centering
\includegraphics[width=0.4\textwidth]{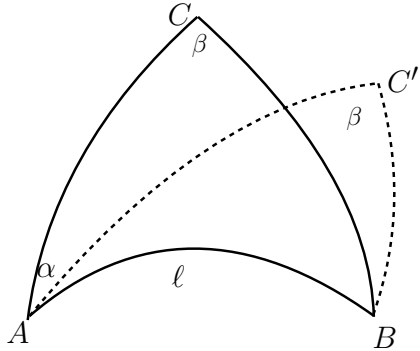}
 \caption{The extremal value of the total angle sum is achieved by an isosceles triangle}
 \label{pic4}
\end{figure}

\begin{proof}
Let $s$ be the sum of the two angles $CAB+CBA$, and let $\alpha$ be the angle $CAB$.  Using the spherical cosine law, we have
\[
\cos \beta=-\cos \alpha \cos (s-\alpha) + \sin \alpha \sin (s-\alpha)\cos \ell
\]
Here we treat $s$ as an implicit function of $\alpha$. Differentiating the two sides with respect to $\alpha$, we get
\[
\begin{split}
0=&\sin \alpha \cos (s-\alpha) +\cos\alpha \sin(s-\alpha)\cos \ell \\ &+(s'-1)[\cos \alpha \sin (s-\alpha) + \sin \alpha \cos (s-\alpha)\cos \ell]
\end{split}
\]
The critical point of $s$ is obtained when $s'=0$, which corresponds to
$$
(1-\cos \ell)\sin (2\alpha-s)=0.
$$
With the conditions $\alpha>0$ and $s-\alpha>0$, we obtain that the critical point satisfies $\alpha=s/2$, i.e. it is an isosceles triangle. 

On the other hand, by analyzing the sign of
$$
s'=1-\frac{\sin \alpha \cos (s-\alpha) +\cos\alpha \sin(s-\alpha)\cos \ell }{\cos \alpha \sin (s-\alpha) + \sin \alpha \cos (s-\alpha)\cos \ell}
$$
we obtain the following two cases:
\begin{itemize}
\item When $\alpha_{0}=s/2>\pi/2$, we have $s'>0$ when $\alpha<\alpha_{0}$ and $s'<0$ when $\alpha>\alpha_{0}$, hence it is a maximum point.
\item When $\alpha_{0}=s/2<\pi/2$, the signs are reversed and hence it is a minimum point for $s$.
\end{itemize}
\end{proof}

\begin{remark}
When $\cos \beta=\cos \ell$, the critical point corresponds to the triangle with angles $\beta, \pi/2, \pi/2$ (so the two extremal triangles are identical). However in this case this point is neither minimum nor maximum, as the family of triangles are given by any combination of $\alpha=\pi/2$ and $s-\alpha\in (0,\pi)$. The proofs of lemma 1 and lemma 2 do not need to use this fact. 
\end{remark}

\section{Proof of Theorem}\label{s:full}
Now we give the proof of the full theorem.
\begin{proof}[Proof of Theorem 1]
When $\alpha=\beta$, we combine Lemma~\ref{lemma1} and Lemma~\ref{lemma2}, and see that the only possibility is the gluing of two footballs. 

When $\alpha\neq \beta$, by replacing the statement $\beta_{1}=\beta_{2}$ to 
$$\beta_{1}=\alpha, \beta_{2}=\beta,$$ 
the proof is similar. We list out the steps as follows.

\noindent\textbf{Step 1: The only possible case is $\alpha=\beta_{1}, \beta=\beta_{2}$.}

Assuming $\beta_{1}=\alpha-2\epsilon$ and $\beta_{2}=\beta+2\epsilon$ for some small $\epsilon$. Similar to Lemma~\ref{lemma2}, we look at the extremal case which is achieved when $\ell_{3}=\ell_{4}$ by Lemma~\ref{symmetry}. In this case we get a similar picture as Figure~\ref{pic5} except the one of the two angles $\beta/2$ is replaced by $\alpha/2$. Using the same computation of Napier's analogies, we get that the sum of $\alpha_{1}+\alpha_{2}$ is either strictly bigger than $4\pi$ in the minimal case or strictly smaller than $4\pi$ in the maximal case. Hence this is not possible when $\epsilon\neq 0$.

\noindent\textbf{Step 2: When $\alpha=\beta_{1}, \beta=\beta_{2}$ we get $g_{s}$ for some $s$.}

We then show that when $\alpha=\beta_{1}, \beta=\beta_{2}$ in Figure~\ref{pic1}, the only possible configuration is indeed the gluing of two footballs. 
\begin{figure}[!h]
\centering
\includegraphics[width=0.9\textwidth]{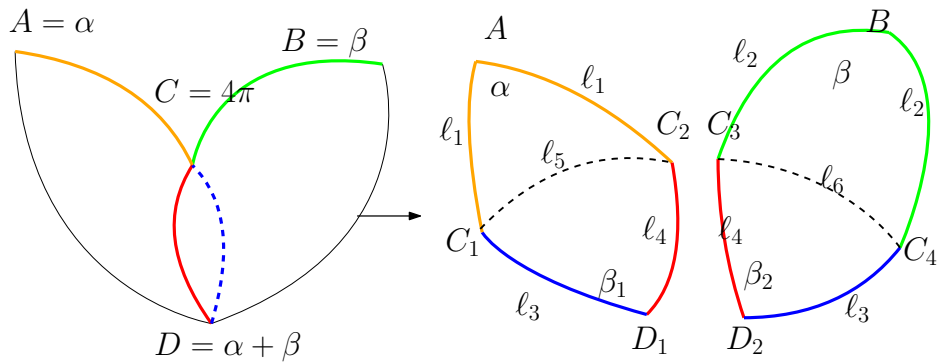}
 \caption{Cutting the manifold into two pieces, each is a surface with 4 sides. A priori they do not need to be spherical bigons as in the glued football construction.}
 \label{pic1}
\end{figure}

When $\alpha \neq \beta$, one no longer has $\ell_{5}= \ell_{6}$ as in the proof of Lemma~\ref{lemma1}. However, by Lemma~\ref{symmetry}, one can see that only the case of $\ell_{3}=\ell_{4}$ would give the possible sum of $4\pi$ (otherwise it would either be bigger or smaller than $4\pi$). Then again by symmetry, we can see that the only possible case is when two spherical triangles $AC_{1}C_{2}$ and $C_{1}C_{2}D_{1}$ (similarly, $B C_{3}C_{4}$ and $C_{3}C_{4}D_{2}$) piece together to a bigon. And this correspond to the gluing of two footballs.
\end{proof}


\begin{thebibliography}{99999}
\bibitem{BMM}
Daniele Bartolucci,  Francesca De Marchis, and Andrea Malchiodi.
\newblock Supercritical conformal metrics on surfaces with conical singularities. 
\newblock {\em Int. Math. Res. Not.}, no. 24, 5625--5643, 2011.


\bibitem{BM}
Daniele Bartolucci, and Andrea Malchiodi. 
\newblock{An improved geometric inequality via vanishing moments, with applications to singular Liouville equations.}
\newblock{\em Communications in Mathematical Physics} 322, no. 2 (2013): 415-452.



\bibitem{CM2}
Alessandro Carlotto, and Andrea Malchiodi.
\newblock Weighted barycentric sets and singular Liouville equations on compact
  surfaces.
\newblock {\em Journal of Functional Analysis}, 262(2):409--450, 2012.

\bibitem{Carlotto}
Alessandro Carlotto.
\newblock{On the solvability of singular Liouville equations on compact surfaces of arbitrary genus.}
\newblock{\em Transactions of the American Mathematical Society} 366, no. 3 (2014): 1237-1256.

\bibitem{CL2}
Chiun-Chuan Chen, and Chang-Shou Lin. 
\newblock{Topological degree for a mean field equation on Riemann surfaces.}
\newblock{\em Communications on Pure and Applied Mathematics} 56, no. 12 (2003): 1667-1727.

\bibitem{CL3}
Chiun-Chuan Chen,  and Chang-Shou Lin. 
\newblock{Mean field equation of Liouville type with singular data: topological degree.}
\newblock{\em Communications on Pure and Applied Mathematics} 68, no. 6 (2015): 887-947.


\bibitem{CCW}
Qing Chen, Xiuxiong Chen, and Yingyi Wu. 
\newblock{The structure of HCMU metric in a K-surface.} 
\newblock{\em International Mathematics Research Notices} 2005, no. 16 (2005): 941-958.

\bibitem{CWWX}
Qing Chen,  Wei Wang, Yingyi Wu, and Bin Xu. 
\newblock{Conformal metrics with constant curvature one and finitely many conical singularities on compact Riemann surfaces.}
\newblock{\em Pacific Journal of Mathematics} 273, no. 1 (2014): 75-100.

\bibitem{Chen}
Xiuxiong Chen.
\newblock{Obstruction to the existence of metric whose curvature has umbilical Hessian in a $K$-surface.}
\newblock{\em Communications in Analysis and Geometry} 8, no. 2 (2000): 267-299.

\bibitem{Dey}
Subhadip Dey.
\newblock Spherical metrics with conical singularities on 2-spheres. 
\newblock {\em Geometriae Dedicata} (2017): 1-9.

\bibitem{Ere2}
Alexandre Eremenko. 
\newblock Metrics of positive curvature with conic singularities on the sphere. 
\newblock {\em Proceedings of the American Mathematical Society},  132.11 (2004): 3349-3355.



\bibitem{Ere1}
Alexandre Eremenko.
\newblock Co-axial monodromy.
\newblock {\em arXiv: 1706.04608v1}, 2017.

\bibitem{EG}
Alexandre Eremenko,  and Andrei Gabrielov. 
\newblock{On metrics of curvature $1$ with four conic singularities on tori and on the sphere.}
\newblock{\em Illinois Journal of Mathematics} 59, no. 4 (2015): 925-947.

\bibitem{EGT}
Alexandre Eremenko, Andrei Gabrielov, and Vitaly Tarasov. 
\newblock Metrics with conic singularities and spherical polygons. 
\newblock {\em Illinois Journal of Mathematics}, 58.3 (2014): 739-755.


\bibitem{EGT2}
Alexandre Eremenko, Andrei Gabrielov, and Vitaly Tarasov. 
\newblock{Metrics with four conic singularities and spherical quadrilaterals.}
\newblock{\em Conformal Geometry and Dynamics of the American Mathematical Society} 20, no. 8 (2016): 128-175.


\bibitem{Ka}
Michael Kapovich.
{\em Branched covers between spheres and polygonal inequalities in simplicial trees.} 
2017.

\bibitem{LT}
Feng Luo, and Gang Tian.
\newblock Liouville equation and spherical convex polytopes.
\newblock {\em Proceedings of the American Mathematical Society},
  116(4):1119--1129, 1992.

\bibitem{MW}
Rafe Mazzeo, and Hartmut Weiss.
\newblock Teichm\"uller theory for conic surfaces.
\newblock In {\em Geometry, Analysis and Probability, In Honor of Jean-Michel
  Bismut}, volume 310 of {\em Progress in Mathematics}, pages 127--164.
  Birkh\"auser Basel, 2017.

\bibitem{MZ}
Rafe Mazzeo, and Xuwen Zhu.
\newblock{Conical metrics on Riemann surfaces, I: the compactified configuration space and regularity.}
 \newblock{\em arXiv:1710.09781}, 2017.


\bibitem{MZ2}
Rafe Mazzeo, and Xuwen Zhu.
\newblock{Conical metrics on Riemann surfaces, II: spherical metrics.}
In preparation.

\bibitem{McOwen}
Robert~C McOwen.
\newblock Point singularities and conformal metrics on Riemann surfaces.
\newblock {\em Proceedings of the American Mathematical Society},
  103(1):222--224, 1988.

\bibitem{MP}
Gabriele Mondello, and Dmitri Panov.
\newblock Spherical metrics with conical singularities on a 2-sphere: angle
  constraints.
\newblock {\em International Mathematics Research Notices},
  2016(16):4937--4995, 2016.
  
\bibitem{MP2}
Gabriele Mondello, and Dmitri Panov.
\newblock Spherical surfaces with conical points: systole inequality and moduli spaces with many connected components.
\newblock{\em arXiv:1807.04373}, 2018


\bibitem{SCLX}
Jijian Song,  Yiran Cheng, Bo Li, and Bin Xu. 
\newblock{Drawing cone spherical metrics via Strebel differentials.}
\newblock{\em International Mathematics Research Notices}, Vol. 00, No. 0, pp. 1-23

\bibitem{SLX}
Jijian Song, Lingguang Li, and Bin Xu. 
\newblock{Cone spherical metrics and stable vector bundles.}
\newblock{\em arXiv:1808.04106}, 2018.

\bibitem{Tr}
Marc Troyanov.
\newblock Prescribing curvature on compact surfaces with conical singularities.
\newblock {\em Transactions of the American Mathematical Society},
  324(2):793--821, 1991.  
  
\bibitem{Tr2}
Marc Troyanov.
\newblock Metrics of constant curvature on a sphere with two conical singularities.
\newblock {\em Lect. Notes Math}, 
1410 (1989): 296-308.

\bibitem{UY}
Masaaki Umehara, and Kotaro Yamada. 
\newblock Metrics of constant curvature 1 with three conical singularities on the 2-sphere. 
\newblock {\em Illinois Journal of Mathematics} 44.1 (2000): 72-94.




\end{thebibliography}
\end{document}